\documentclass[a4paper,12pt]{amsart}
\usepackage{amssymb,amsfonts,amsmath}
\usepackage{ifthen}
\usepackage{graphicx}
\usepackage{float}
\newtheorem{theorem}{Theorem}[section]
\newtheorem{corollary}[theorem]{Corollary}
\newtheorem{lemma}[theorem]{Lemma}
\newtheorem{proposition}[theorem]{Proposition}

\theoremstyle{definition}
\newtheorem{definition}{Definition}[section]

\newtheorem{remark}[theorem]{Remark}

\newcounter{minutes}
\divide\time by 60
\newcounter{hours}
\multiply\time by 60
\addtocounter{minutes}{-\time}
\newcommand{\D}{{\mathbb D}}

\newcommand{\real}{{\operatorname{Re}\,}}
\newcommand{\Log}{{\operatorname{Log}\,}}
\newcommand{\Arg}{{\operatorname{Arg}\,}}
\newcommand{\ds}{\displaystyle}

\begin{document}
%\hfill{File}

\bibliographystyle{amsplain}

\title[$q$-starlike functions of order alpha]%
{A generalization of starlike functions of order alpha}

%%%%%%%% BEGIN TIMESTAMP
\def\thefootnote{}
\footnotetext{ \texttt{\tiny File:~\jobname .tex,
          printed: \number\day-\number\month-\number\year,
          \thehours.\ifnum\theminutes<10{0}\fi\theminutes}
} \makeatletter\def\thefootnote{\@arabic\c@footnote}\makeatother
%%%%%%%% END TIMESTAMP

\author{Sarita Agrawal${}^{\#}$}
%\address{S. Agrawal, Discipline of Mathematics,
%Indian Institute of Technology Indore,
%Indore 452 017, India}
%\email{saritamath44@gmail.com}

\author{Swadesh Kumar Sahoo${}^{*}$}
%\address{S. K. Sahoo, Discipline of Mathematics,
%Indian Institute of Technology Indore,
%Indore 452 017, India}
%\email{swadesh@iiti.ac.in}

\thanks{${}^*$ The corresponding author\\
${}^{\#}$ Sarita Agrawal, Discipline of Mathematics,
Indian Institute of Technology Indore,
Simrol, Khandwa Road, Indore 452 020, India\\
{\em Email: saritamath44@gmail.com}\\
${}^*$ Swadesh Kumar Sahoo, Discipline of Mathematics,
Indian Institute of Technology Indore,
Simrol, Khandwa Road, Indore 452 020, India\\
{\em Email: swadesh@iiti.ac.in}
}

%\subjclass[2010]{28A25; 30B10; 30C45; 30C50; 30C55; 33B10; 33C05; 33D15; 39A13; 39A70; 40A20; 44A05; 44A15; 45P05; 46G05; 47B38; 47B39}
%\keywords{starlike functions, $q$-starlike functions, order of starlikeness and $q$-starlikeness, 
%$q$-difference operator, $q$-difference equation, $q$-integrals, Bieberbach's conjecture, 
%infinite product, uniform convergence, basic hypergeometric functions, Herglotz representation,
%bijection map, probability measure. \\
%$%{}^{\mathbf{*}}
%^\dagger$ {\tt Corresponding author}
%}

\begin{abstract}
For every $q\in(0,1)$ and $0\le \alpha<1$ we define a class of analytic functions, the so-called 
$q$-starlike functions of order $\alpha$, on the open unit disk. We study this class of functions
and explore some inclusion properties with the well-known class of starlike functions of order
$\alpha$. The paper
is also devoted to the discussion on the Herglotz representation formula for analytic functions 
$zf'(z)/f(z)$ when $f(z)$ is $q$-starlike of order $\alpha$. As an application we also discuss 
the Bieberbach conjecture problem for the $q$-starlike functions of order $\alpha$.
%We also
%consider the concept of $q$-integrals and as a result we find a function that maximizes the 
%coefficients of power series of $q$-starlike functions of order $\alpha$. 
Further application includes the study of the order of $q$-starlikeness of the well-known basic 
hypergeometric functions introduced by Heine.

%45P05; 
\end{abstract}

%\begin{keywords}starlike functions, $q$-starlike functions, order of starlikeness and $q$-starlikeness, 
%$q$-difference operator, $q$-difference equation, $q$-integrals, Bieberbach's conjecture, 
%infinite product, uniform convergence, basic hypergeometric functions, Herglotz representation,
%bijection map, probability measure.
%\end{keywords}
%\begin{classcode}28A25; 30B10; 30C45; 30C50; 30C55; 33B10; 33C05; 33D15; 39A13; 39A70; 40A20; 44A05;
%44A15; 45P05; 46G05; 47B38; 47B39.
%\end{classcode}
\maketitle
\pagestyle{myheadings}
\markboth{S. Agrawal and S. K. Sahoo}{The $q$-starlike functions of order alpha}

\section{Introduction and Main Results}
In view of the well-known Riemann mapping theorem in classical complex analysis,
the unit disk $\mathbb{D}=\{z\in\mathbb{C}:\,|z|<1\}$ is usually considered 
as a standard domain. The analytic functions such as convex, starlike, and close-to-convex 
functions defined in the unit disk have been extensively studied and found numerous 
applications to various problems in complex analysis and related topics. 
Part of this development is the study of subclasses of the class of 
univalent functions, more general than the classes of 
convex, starlike, and close-to-convex functions. Analytic and geometric characterizations of such functions 
are of quite interesting to all function theorists in general.
Background knowledge in this theory can be found from standard books
(see for instance \cite{Dur83,Goo83}). 

In 1916, Bieberbach first posed a conjecture on the coefficient estimate of univalent
functions. This conjecture was a long standing open problem in univalent function theory and was a 
challenge to all mathematicians.
In this regard a lot of methods and concepts were developed. One of the important concepts 
is the {\em Herglotz representation theorem} for univalent functions with positive real part. 
Initially, the Bieberbach conjecture
was proved for first few coefficients of univalent functions. Then the conjecture was considered in 
many special cases. In one direction, it was considered for certain subclasses of univalent functions 
like starlike, convex, close-to-convex, typically real functions etc. The concept of order for the 
starlike and convex was also introduced, which are the subclasses of the class of starlike and convex 
functions respectively, and
the conjecture was proved in these subclasses. In other direction, discussion on 
many conjectures, namely, the Zalcman conjecture, the Robertson conjecture, the Littlewood-Paley conjecture, etc. were investigated to prove the Bieberbach conjecture. 
Finally, the full conjecture for univalent functions was settled down by L. de Branges in 1985 \cite{deB85}.

In 1990, Ismail et. al. \cite{IMS90} introduced a link between starlike functions 
and the $q$-theory by introducing a $q$-analog of the starlike functions. We call these
functions as $q$-starlike functions. 
They proved the Bieberbach conjecture for the $q$-starlike functions through the 
Herglotz representation theorem for these functions. 
In this connection, we aim to introduce the concept of order of $q$-starlikeness
and prove the Bieberbach conjecture.
In particular, we also discuss several other basic properties on the order of 
$q$-starlike functions.
%As a result we also consider the shifted basic 
%hypergeometric functions introduced by Heine \cite{Hei46} and study its $q$-analog of 
%order of starlikeness.
%Basic background knowledge on the order of starlikeness of the well-known Gauss hypergeometric
%functions can be found in \cite{HPV10,Kus02,MM90,Pon97,PV01,Sil93}.  

We now collect some standard notations and basic definitions used in the sequel.
%Our notations are quite standard.
We denote by $\mathcal{H}(\D)$, the set of all analytic (or holomorphic) 
functions in $\D$. We use the symbol $\mathcal{A}$ for the class of functions 
$f \in \mathcal{H}(\D)$ with 
the standard normalization $f(0)=0=f'(0)-1$. This means that the functions 
$f\in\mathcal{A}$ have the power series representation of the form $z+\sum_{n=2}^\infty a_nz^n$. 
The principal value of the logarithmic function $\log z$ for $z\neq 0$ is denoted by 
$\Log z:=\ln |z|+i \Arg(z)$, where $-\pi\le \Arg(z)<\pi$.

For $0<q<1$, {\em the $q$-difference operator} denoted as $D_qf$ is defined by
the equation
\begin{equation}\label{sec1-eqn1}
(D_qf)(z)=\frac{f(z)-f(qz)}{z(1-q)},\quad z\neq 0, \quad (D_qf)(0)=f'(0).
\end{equation}
The operator $D_qf$ plays an important role in the theory of basic hypergeometric series 
(see \cite{AS14,And74,Fin88,Sla66}); see also Section~4 in this paper.
It is evident that, when $q\to 1^{-}$, the difference operator $D_qf$ converges to the ordinary 
differential operator $Df=df/{dz}=f'$.

A function $f\in \mathcal{A}$ is called starlike of order $\alpha$, $0\le \alpha<1$, if 
$${\rm Re}\,\left(\frac{zf'(z)}{f(z)}\right)>\alpha,\quad z\in \mathbb{D}.
$$
We use the notation $\mathcal{S}^*(\alpha)$ for the class of starlike functions of order
$\alpha$. Set $\mathcal{S^*}:=\mathcal{S^*}(0)$, the class of all starlike functions.

One way to generalize the starlike functions of order $\alpha$ is to replace the derivative function $f'$ by the
$q$-difference operator $D_qf$ and replace the right-half plane $\{w:\,{\rm Re}\,w>\alpha\}$ by a suitable 
domain in the above definition of the starlike functions of order $\alpha$. The
appropriate definition turned out to be the following:
%By \cite{IMS90}, a function $f\in\mathcal{A}$ is said to belong to the class $\mathcal{S}^*_q$ if
%$$\left|\frac{z(D_qf)(z)}{f(z)}-\frac{1}{1-q}\right|\leq \frac{1}{1-q}, \quad z\in \mathbb{D} .
%$$
\begin{definition}\label{main:defn}
A function $f\in\mathcal{A}$ is said to {\em belong to the class $\mathcal{S}^*_q(\alpha)$}, $0\le \alpha<1$, if
$$\left|\frac{\ds\frac{z(D_qf)(z)}{f(z)}-\alpha}{1-\alpha}-\frac{1}{1-q}\right|\leq \frac{1}{1-q}, \quad z\in \mathbb{D}.
$$
\end{definition}
The following is the equivalent form of Definition~\ref{main:defn}. 
\begin{equation}\label{main=defn}
f\in\mathcal{S}_q^*(\alpha) \iff \left|\frac{z(D_qf)(z)}{f(z)}-\frac{1-\alpha q}{1-q}\right|
\le \frac{1-\alpha}{1-q}.
\end{equation}
%This form of the definition is useful in Section~2 and an equivalent form of this 
%definition is also considered in Section~4. 
Observe that as $q\to 1^{-}$ the closed disk $|w-(1-q)^{-1}|\le (1-q)^{-1}$ becomes the right-half plane
and the class $\mathcal{S}^*_q(\alpha)$ reduces to $\mathcal{S}^*(\alpha)$, $0\le \alpha<1$.
In particular, when $\alpha=0$, the class $\mathcal{S}^*_q(\alpha)$ coincides with the class 
$\mathcal{S}^*_q:=\mathcal{S}^*_q(0)$, which was
first introduced by Ismail et. al. \cite{IMS90} in 1990 and later (also recently) 
it has been considered in 
\cite{AS14,RS12,Ro92,SS14}.
%In \cite{IMS90}, it is proved that 
%$$\mathcal{S}^*=\bigcap_{0<q<1} \mathcal{S}^*_q.
%$$
In words we call $\mathcal{S}^*_q(\alpha)$, the class of {\em $q$-starlike functions of order $\alpha$}.
%If $f$ is starlike of order $\alpha$, then $\sigma(f)\ge \alpha$. 
%Our main objective in this section is to study the order of $q$-starlikeness for the shifted basic
%hypergeometric functions $z\Phi[a,b,c,q,z]$ which generalized the result \cite[Theorem~1.1]{Kus02} of K\"ustner.
%We now state and prove our main result concerning the order of the $q$-starlikeness of 
%the shifted basic hypergeometric functions $z\Phi[a,b;c;q,rz]$, for real parameters $a,b$ and $c$.

The main objective in this paper is to prove the following theorems. The first main theorem describes 
the {\em Herglotz Representation} for functions belonging to the class $\mathcal{S}^*_q(\alpha)$ in 
the form of a Poisson-Stieltjes integral (see Herglotz Representation Theorem for analytic functions 
with positive real part in \cite[pp.~22]{Dur83}).

\begin{theorem}\label{thm2}
Let $f\in \mathcal{A}$. Then $f\in\mathcal{S}^*_q(\alpha)$ if and only if there exists 
a probability measure $\mu$ supported on the unit circle such that
$$\frac{zf'(z)}{f(z)}=1+\int_{|\sigma|=1}\sigma z F_{q, \alpha}^{'}(\sigma z)\rm{d}\mu(\sigma)
$$
where
\begin{equation}\label{MainThm1:eq}
F_{q,\alpha}(z)=\ds \sum_{n=1}^\infty \frac{(-2)\left(\ln \frac{q}{1-\alpha(1-q)}\right)}{1-q^n}z^n, \quad z\in \D .
\end{equation}
\end{theorem}

\begin{remark}
When $q$ approaches $1$, Theorem~\ref{thm2} leads to the Herglotz Representation Theorem for 
starlike functions of order $\alpha$ (see for instance \cite[Problem~3, pp. 172]{Goo83}).
\end{remark}

Our second main theorem concerns about the Bieberbach conjecture problem for functions in $\mathcal{S}^*_q(\alpha)$. 
The extremal
function is also explicitly obtained in terms of exponential of the function $F_{q, \alpha}(z)$.
%(see Figure~\ref{Fx}, 
%%Figures~\ref{F1by21by2x} and \ref{F5by61by2x}, 
%when $z=x$ is real).
This exponential form generalizes the Koebe function $k_\alpha(z)=z/(1-z)^{2(1-\alpha)}$, $z\in \mathbb{D}$.
That is, when $q\to 1^{-}$, the exponential form $G_{q,\alpha}(z):=z\, \exp [F_{q, \alpha}(z)]$ representing the 
extremal function for the class $\mathcal{S}^*_q(\alpha)$ turns into the Koebe function $k_\alpha(z)$.

\begin{theorem}\label{sec2-thm7}
Let 
\begin{equation}\label{MainThm2:eq}
G_{q, \alpha}(z):=z\, \exp [F_{q, \alpha}(z)]
=z+\ds \sum_{n=2}^\infty c_n z^n.
\end{equation}
Then $G_{q, \alpha}\in \mathcal{S}^*_q(\alpha)$. Moreover, if $f(z)=z+\sum_{n=2}^\infty a_n z^n\in 
\mathcal{S}^*_q(\alpha)$, then $|a_n|\le c_n$ with equality holding for all $n$ if and only if 
$f$ is a rotation of $G_{q, \alpha}$.
\end{theorem}

\begin{remark}
When $q$ approaches $1$, Theorem~\ref{sec2-thm7} leads to the Bieberbach conjecture for
starlike functions of order $\alpha$ (see for instance \cite[Theorem~2, pp. 140]{Goo83}).
\end{remark}

%The Herglotz representation theorem and the Bieberbach conjecture problem for $q$-starlike 
% functions are already discussed by Ismail in \cite{IMS90}. These are the fundamental theorems in $q$-theory. 
% Also these theorems are known for starlike functions of order $\alpha$. So it is worthwhile to study these theorems 
% for $q$-starlike functions of order $\alpha$.
Motivation behind this comes from the work of Ismail et. al., where the $q$-analog of starlike functions was 
introduced in 1990 (see \cite{IMS90}).
The $q$-theory has important role in special functions and quantum physics 
(see for instance \cite{And74,Ern02,Fin88,KC02,Kir95,Sla66}). For updated research work
in function theory related to $q$-analysis, readers can refer \cite{AS14,IMS90,RS12,Ro92,SS14}.
In \cite{IMS90}, the authors have obtained the Herglotz representation for functions of the class 
$\mathcal{S}_q^*$ in the following form:

\medskip
\noindent
{\bf Theorem~A.} \cite[Theorem~1.15]{IMS90}
{\em Let $f\in \mathcal{A}$. Then $f\in\mathcal{S}^*_q$ if and only if there exists 
a probability measure $\mu$ supported on the unit circle such that
$$\frac{zf'(z)}{f(z)}=1+\int_{|\sigma|=1}\sigma z F_{q}^{'}(\sigma z)\rm{d}\mu(\sigma)
$$
where
$$F_{q}(z)=\ds \sum_{n=1}^\infty \frac{-2\ln q}{1-q^n}z^n, \quad z\in \D .
$$
}

\medskip
\noindent
Also they have proved the Bieberbach conjecture problem for $q$-starlike functions in the following form:

\medskip
\noindent
{\bf Theorem~B.} \cite[Theorem~1.18]{IMS90}
{\em Let 
$$G_{q}(z):=z\, \exp [F_{q}(z)]
%=z\, \exp \Big[\ds \sum_{n=1}^\infty (1-\alpha)\frac{2\ln q}{q^n-1}z^n\Big]
=z+\ds \sum_{n=2}^\infty c_n z^n.
$$
Then $G_{q}\in \mathcal{S}^*_q$. Moreover, if $f(z)=z+\sum_{n=2}^\infty a_n z^n\in 
\mathcal{S}^*_q$, then $|a_n|\le c_n$ with equality holding for all $n$ if and only if 
$f$ is a rotation of $G_{q}$.
}

\medskip
\noindent
\begin{remark}
It is remarkable that when $\alpha=0$, Theorem~\ref{thm2} and Theorem~\ref{sec2-thm7} respectively coincides 
with Theorem~A and Theorem~B.
\end{remark}

%The Herglotz representation theorem and the Bieberbach conjecture problem are two fundamental theorems in 
%function theory. These theorems were studied only for functions belonging to the class $\mathcal{S}_q^*$ by 
%Ismail et. al. So it is worthwhile to study these problems for the class $\mathcal{S}_q^*(\alpha)$.

Section~2 is devoted for basic interesting properties of the class $\mathcal{S}^*_q(\alpha)$, which are 
used in the proof of main theorems. In Section~3, we prove our main results.
%characterize functions $f\in \mathcal{S}_q^*(\alpha)$ in such a way that the functions $zf'(z)/f(z)$
%can be expressed in terms of an integral of a convergent $q$-series with respect to a probability measure supported on
%the unit circle $\partial\mathbb{D}:=\{z\in\mathbb{C}:\,|z|=1\}$. As a result this produces 
%functions in $\mathcal{S}_q^*(\alpha)$, whose coefficients have maximum moduli.
%The corresponding extremal function with its graph is also explicitly obtained.
The order of $q$-starlikeness of the basic hypergeometric functions is discussed in Section~4. 
Finally, we focus on concluding remarks with few questions in Section~5 for future research in this direction.  

\section{Properties of the class $\mathcal{S}^*_q(\alpha)$}
As a matter of fact the following proposition says that a function $f$ in $\mathcal{S}^*_q(\alpha)$ 
can be obtained in terms of a function $g$ in $\mathcal{S}^*_q$. The proof is obvious and it follows
from the definition of $\mathcal{S}^*_q(\alpha)$, $0\le \alpha< 1$.

\begin{proposition}\label{sec1-prop1}
Let $f \in \mathcal{S}^*_q(\alpha)$. Then there exists a unique function $g \in \mathcal{S}^*_q$ such that 
\begin{equation}\label{prop1-eqn}
\frac{\ds\frac{z(D_qf)(z)}{f(z)}-\alpha}{1-\alpha}=\frac{z(D_qg)(z)}{g(z)} 
~~\mbox{ or }~~
\frac{f(qz)-\alpha qf(z)}{(1-\alpha)f(z)}=\frac{g(qz)}{g(z)}.
\end{equation}
holds. Similarly, for a given function $g\in \mathcal{S}_q^*$ there exists 
$f\in \mathcal{S}_q^*(\alpha)$ satisfying the above relation. Uniqueness follows trivially.
\end{proposition}

Next we present a easy characterization of functions in the class $\mathcal{S}^*_q(\alpha)$.
This shows that if $f\in \mathcal{S}^*_q(\alpha)$ then $f(z) = 0$ implies $z = 0$, otherwise
$f(qz)/f(z)$ would have a pole at a zero of $f(z)$ with least nonzero modulus.
\begin{theorem}\label{sec2-thm1}
Let $f\in \mathcal{A}$. Then $f\in\mathcal{S}^*_q(\alpha)$ if and only if
$$\left|\frac{f(qz)}{f(z)}-\alpha q\right|\leq 1-\alpha, \quad z\in \mathbb{D}.
$$
\end{theorem}
\begin{proof}
The proof can be easily obtained from the fact
$$
\frac{z(D_qf)(z)}{f(z)}= \left(\frac{1}{1-q}\right)\left(1-\frac{f(qz)}{f(z)}\right)
$$
and the definition of $\mathcal{S}^*_q(\alpha)$.
\end{proof}

The next result is an immediate consequence of Theorem~\ref{sec2-thm1}.
\begin{corollary}
The class $\mathcal{S}^*_q(\alpha)$ satisfies the inclusion relation 
$$\bigcap_{q<p<1}\mathcal{S}^*_p(\alpha)\subset \mathcal{S}^*_q(\alpha)
~~\mbox{ and }~~
\bigcap_{0<q<1}\mathcal{S}^*_q(\alpha) = \mathcal{S}^*(\alpha).
$$ 
\end{corollary}
\begin{proof}
The inclusions 
$$\bigcap_{q<p<1}\mathcal{S}^*_p(\alpha)\subset \mathcal{S}^*_q(\alpha)
~~\mbox{ and }~~
\bigcap_{0<q<1}\mathcal{S}^*_q(\alpha) \subset \mathcal{S}^*(\alpha)
$$ 
clearly hold. It remains to show that 
$$\mathcal{S}^*(\alpha) \subset \bigcap_{0<q<1}\mathcal{S}^*_q(\alpha)
$$
holds. For this, we let $f\in \mathcal{S}^*(\alpha)$. Then it is enough
to show that $f\in \mathcal{S}^*_q(\alpha)$ for all $q\in (0,1)$.
Since $f\in \mathcal{S}^*(\alpha)$ there exists a unique $g\in \mathcal{S}^*$
satisfying
$$\frac{\ds\frac{zf'(z)}{f(z)}-\alpha}{1-\alpha}=\frac{zg'(z)}{g(z)},\quad |z|<1.
$$
Since $\mathcal{S}^*=\cap_{0<q<1}\mathcal{S}^*_q$, it follows that $g\in \mathcal{S}^*_q$ for all $q\in(0,1)$. 
Thus, by Proposition~\ref{sec1-prop1} there exists a unique $h\in \mathcal{S}_q^*(\alpha)$ satisfying the 
identity (\ref{prop1-eqn}) with $h(z)=f(z)$. The proof now follows immediately.
\end{proof}

We now define two sets and proceed to prepare some basic results which are being used 
to prove our main results in this section. They are
$$B_q=\{g:g\in \mathcal{H}(\D),~g(0)=q \mbox{ and } g:\D \to \D\}
~~\mbox{ and }~~
B_q^0=\{g:g\in B_q \mbox{ and } 0\notin g(\D) \}.
$$
\begin{lemma}{\label{lm2}}
If $h\in B_q$ then the infinite product 
$\prod_{n=0}^\infty \{((1-\alpha)h(zq^n)+\alpha q)/q\}$ converges uniformly on compact subsets of $\D$.
\end{lemma}
\begin{proof}
We set $(1-\alpha)h(z)+\alpha q=g(z)$. Since $h \in B_q$, it easily follows that $g \in B_q$.
By \cite[Lemma~2.1]{IMS90}, the conclusion of our lemma follows.
\end{proof}

\begin{lemma}{\label{lm3}}
If $h\in B_q^0$ then the infinite product $\prod_{n=0}^\infty \{((1-\alpha)h(zq^n)+\alpha q)/q\}$ converges 
uniformly on compact subsets of $\D$ to a nonzero function in $\mathcal{H}(\D)$ with no zeros. Furthermore, the function 
\begin{equation}\label{eq3}
f(z)=\frac{z}{\prod_{n=0}^\infty \{((1-\alpha)h(zq^n)+\alpha q)/q\}}
\end{equation}
belongs to $\mathcal{S}^*_q(\alpha)$ and $h(z)=((f(qz)/f(z))-\alpha q)/(1-\alpha)$.
\end{lemma}
\begin{proof}
The convergence of the infinite product is proved in Lemma \ref{lm2}. 
Since $h\in B_q^0$, we have $h(z)\neq 0$ in 
$\D$ and the infinite product does not vanish in $\D$. Thus, the function $f\in \mathcal{A}$ and
we find the relation
$$\frac{f(qz)}{f(z)}=(1-\alpha)h(z)+\alpha q, ~~\mbox{ equivalently }~~ 
\frac{\ds\frac{f(qz)}{f(z)}-\alpha q}{1-\alpha}=h(z) .
$$
Since $h\in B_q^0$, we get $f\in \mathcal{S}^*_q(\alpha)$ and the proof of our lemma is complete.
\end{proof}
We define two classes $B_{q,\alpha}$ and $B_{q,\alpha}^0$ by
$$B_{q,\alpha}=\left\{g: g\in \mathcal{H}(\D),~g(0)=\frac{q}{1-\alpha(1-q)}\mbox{ and } g:\D \to \D \right\}
$$
and
$$
B_{q,\alpha}^0=\{g:g\in B_{q,\alpha} \mbox{ and } 0\notin g(\D) \}.
$$

\begin{lemma}\label{lm}
A function $g\in B_{q,\alpha}^0$ if and only if it has the representation
\begin{equation}\label{eq4}
g(z)=\exp\left\{\left(\ln\frac{q}{1-\alpha(1-q)}\right)p(z)\right\}, 
\end{equation}
where $p(z)$ belongs to the class
$$
\mathcal{P}=\{p: p\in \mathcal{H}(\D), p(0)=1 \mbox{ and } \real \{p(z)\}\ge 0 \mbox{ for } z\in\D\}.
$$
\begin{proof}
For $g\in B_{q,\alpha}^0$, define the function $L(z)=\Log g(z)$. Then it is easy to show
that the function $p(z)=\ds\frac{L(z)}{\ln \frac{q}{1-\alpha(1-q)}}\in\mathcal{P}$ and satisfies (\ref{eq4}). 
Conversely, if $g$ is given by (\ref{eq4}), then it is obvious that $g\in B_{q,\alpha}^0$.
\end{proof}
\end{lemma}
\begin{theorem}\label{thm1}
The mapping $\rho:\mathcal{S}^*_q(\alpha) \to B_q^0$ defined by
$$\rho(f)(z)=\frac{\ds\frac{f(qz)}{f(z)}-\alpha q}{1-\alpha}
$$
is a bijection.
\end{theorem}
\begin{proof}
%In view of Lemma \ref{lm2} and Lemma \ref{lm3}, we remain to show that $\rho$ is one-one. 
%For $f_1, f_2 \in \mathcal{S}^*_q(\alpha)$ we let $\rho(f_1)=\rho(f_2)$, i.e.
%$$\frac{\ds\frac{f_1(qz)}{f_1(z)}-\alpha q}{1-\alpha} 
%= \frac{\ds\frac{f_2(qz)}{f_2(z)}-\alpha q}{1-\alpha} \in B_q^0,\quad |z|<1 ,
%$$
%which implies
%$$\frac{f_1(qz)}{f_1(z)}=\frac{f_2(qz)}{f_2(z)} .
%$$
%It follows that the function $\psi(z)=f_1(z)/f_2(z)$ satisfies the recurrence relation 
%$$ \psi(z)=\psi(zq^n), \quad n=1, 2, ..., z \in \D .
%$$
%Therefore, $\psi(z)$ must be a constant. This constant must be $1$, since $f_1^{'}(0)=f_2^{'}(0)=1$.
%This completes the proof of our theorem.
%Since $\rho(f)(z)=h(z)$ is in $B_q^0$, by Lemma~\ref{lm3}, $f(z)\in \mathcal{S}^*_q(\alpha)$ and 
%it has the representation (\ref{eq3}). Let $\sigma(h)(z)=f(z)$, where $f(z)$ has the same representation 
%as in Lemma~\ref{lm3}.
For $ h \in B_q^0 $, define a mapping $\sigma:\,B_q^0 \to \mathcal{A}$ by 
$$ \sigma(h)(z)=\frac{z}{\prod_{n=0}^\infty \{((1-\alpha)h(zq^n)+\alpha q)/q\}}.
$$
It is clear from Lemma~\ref{lm3} that $\sigma(h) \in \mathcal{S}^*_q(\alpha)$ and $(\rho\circ \sigma)(h)=h$. Considering the composition mapping $\sigma\circ \rho$ we compute that
%$\sigma\circ \rho$ and $\rho\circ \sigma$. 
$$(\sigma\circ \rho)(f)(z)=\frac{z}{\prod_{n=0}^\infty \{(f(zq^{n+1})/qf(zq^n)\}}=\frac{z}{z/f(z)}=f(z).
$$
%and 
%$$(\rho\circ \sigma)(h)=\rho(\sigma(h))=\rho(f)=h.
%$$
Hence $\sigma\circ \rho$ and $\rho\circ \sigma$ are identity mappings and $\sigma$ 
is the inverse of $\rho$, i.e. the map $\rho(f)$ is invertible. Hence $\rho(f)$ is a bijection. This completes the proof of our theorem.
\end{proof}
\section{Proof of the main theorems}
%One of the crucial supplementary results in supporting the materials for the proof of the Bieberbach 
%conjecture problem for functions in $\mathcal{S}^*_q(\alpha)$ is the following. This is one kind of 
%{\em Herglotz Representation Theorem} for the class $\mathcal{S}^*_q(\alpha)$ in the form of a Poisson-Stieltjes integral.
%(see Herglotz Representation Theorem for
%analytic functions with positive real part in \cite[pp.~22]{Dur83}).
%\begin{figure}[H]
%\begin{minipage}[b]{0.5\textwidth}
%\includegraphics[width=7cm]{F1by21by2x.pdf}
%\centering{Graph of $F_{1/2,1/2}(x)$, $0\le x<1$} %\label{F1by21by2x}
%\end{minipage}
%\begin{minipage}[b]{0.45\textwidth}
%\includegraphics[width=7cm]{F5by61by2x.pdf}
%\centering{Graph of $F_{5/6,1/2}(x)$, $0\le x<1$} %\label{F5by61by2x}
%\end{minipage}
%\vskip 0.3cm
%\caption{Graphs of the real functions $F_{q,\alpha}(x)$ for $0\le x<1$.}\label{Fx}
%\end{figure}

This section is devoted to the proofs of main theorems using the supplementary results proved in Section~2.

\begin{figure}[H]
\begin{minipage}[b]{0.5\textwidth}
\includegraphics[width=5.5cm]{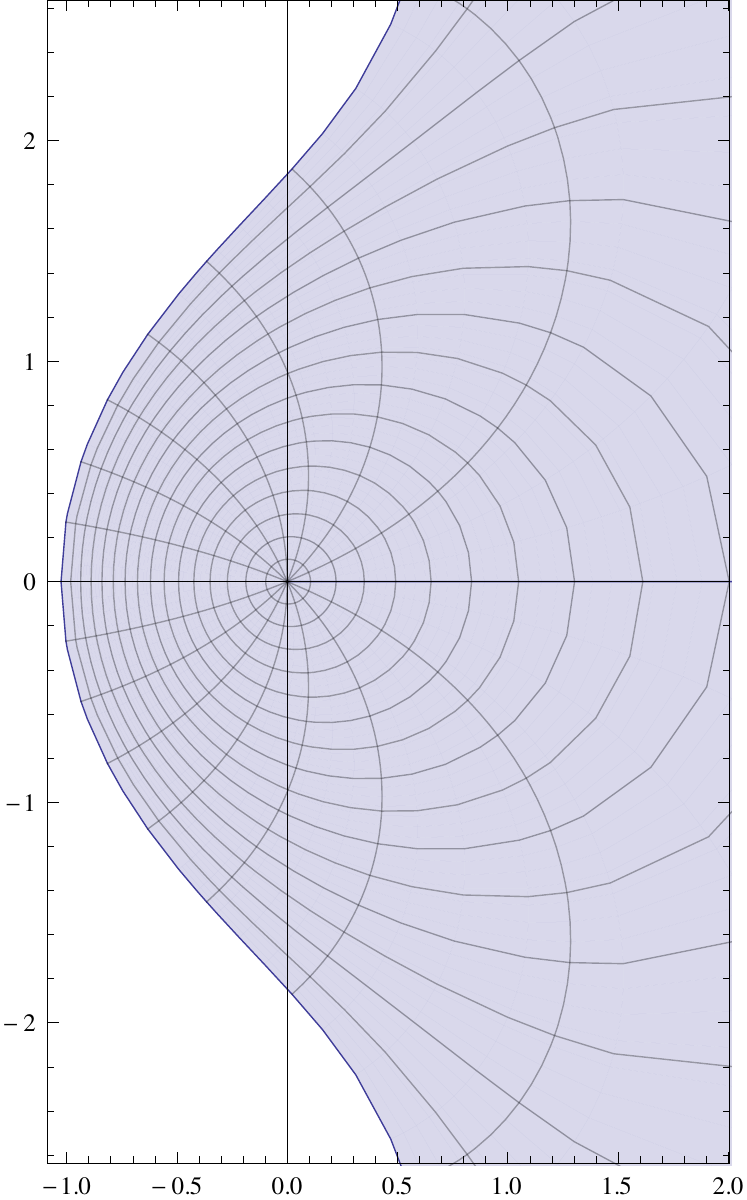}
\vskip 0.1cm \hskip 0.4cm
Graph of $F_{1/2,1/2}(z)$, $|z|<1$ %\label{F1by21by2z}
\end{minipage}
\begin{minipage}[b]{0.35\textwidth}
\includegraphics[width=5.5cm,height=8.9cm]{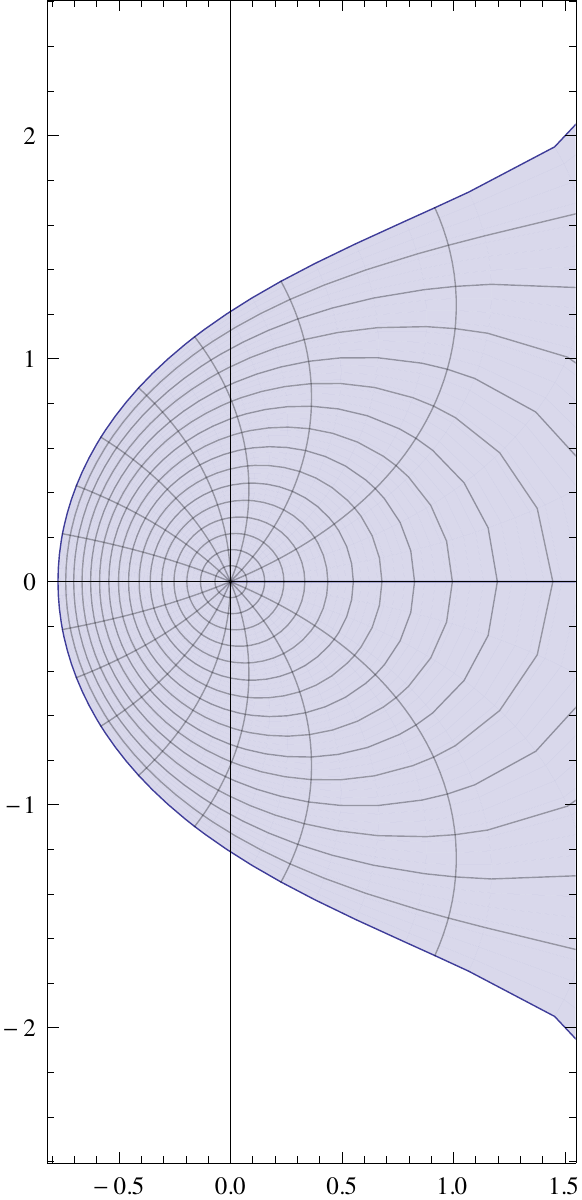}
\vskip 0.1cm \hskip 0.4cm
Graph of $F_{5/6,1/2}(z)$, $|z|<1$ %\label{F5by61by2z}
\end{minipage}
\vskip 0.3cm
\caption{Graphs of the complex functions $F_{q,\alpha}(z)$ for $|z|<1$.}\label{Fz}
\end{figure}

\begin{proof}[\bf Proof of Theorem~\ref{thm2}]
For $0<q<1$ and $0\le \alpha<1$, let $F_{q, \alpha}$ be defined by (\ref{MainThm1:eq}). 
Geometry of $F_{q, \alpha}$ is described in Figure~\ref{Fz} for different ranges
over the parameters $q$ and $\alpha$.
Suppose that $f\in \mathcal{S}^*_q(\alpha)$.
Then by Theorem~\ref{thm1} and Lemma~\ref{lm3}, it is clear that $f(z)$ has the 
representation (\ref{eq3}) with $h\in B_q ^0$. The logarithmic derivative of $f$ gives
\begin{equation}\label{eq5}
\frac{zf'(z)}{f(z)}=1-\sum_{n=0}^\infty \frac{(1-\alpha)zq^nh'(zq^n)}{(1-\alpha)h(zq^n)+\alpha q}.
\end{equation}
Now, let us assume that
$$
g(z)=\frac{(1-\alpha)h(z)+\alpha q}{1-\alpha(1-q)}.
$$
Clearly, $g\in B_{q,\alpha}^0$ and hence Lemma~\ref{lm} guarantees that $g(z)$ has the 
representation (\ref{eq4}). Taking the logarithmic derivative of $g$ we have
\begin{equation}\label{eq6}
\frac{zg'(z)}{g(z)}=\left(\ln\frac{q}{1-\alpha(1-q)}\right)zp'(z),
\end{equation}
where $\real \{p(z)\}\ge 0$. By Herglotz representation of $p(z)$, there exists a probability measure 
$\mu$ supported on the unit circle $|\sigma|=1$ such that
\begin{equation}\label{eq7}
zp'(z)=\int_{|\sigma|=1}2\sigma z(1-\sigma z)^{-2}d\mu(\sigma).
\end{equation}
Using (\ref{eq6}) and (\ref{eq7}) in (\ref{eq5}), we have
\begin{eqnarray*}
\frac{zf'(z)}{f(z)}&=& 1-2\left(\ln\frac{q}{1-\alpha(1-q)}\right)\sum_{n=0}^\infty 
\int_{|\sigma|=1}\sigma z q^n(1-\sigma zq^n)^{-2}d\mu(\sigma)\\
&=& 1-2\left(\ln\frac{q}{1-\alpha(1-q)}\right)\int_{|\sigma|=1}
\left\{\sum_{n=0}^\infty\sum_{m=1}^\infty m{\sigma}^m z^m q^{mn}\right\}d\mu(\sigma)\\
&=& 1-2\left(\ln\frac{q}{1-\alpha(1-q)}\right)\int_{|\sigma|=1}\left\{\sum_{m=1}^\infty 
m{\sigma}^m z^m\frac{1}{1-q^m}\right\}d\mu(\sigma)\\
&=&1+\int_{|\sigma|=1}\sigma z F_{q, \alpha}^{'}(\sigma z)\rm{d}\mu(\sigma).
\end{eqnarray*}
%By Proposition~\ref{sec1-prop1}, there exists a function $ f(z)\in\mathcal{S}^*_q$ such that
%$$\frac{\ds\frac{z(D_q g)(z)}{g(z)}-\alpha}{1-\alpha}=\frac{z(D_qf)(z)}{f(z)},
%$$
%equivalently, we obtain
%$$\frac{(D_q g)(z)}{g(z)}=(1-\alpha)\frac{(D_q f)(z)}{f(z)}+\frac{\alpha}{z},\quad |z|<1,~0<q<1.
%$$
%As $q\to 1^{-}$, we get
%\begin{equation}\label{thm7-eqn}
%\frac{zg^{'}(z)}{g(z)}=(1-\alpha)\frac{zf^{'}(z)}{f(z)}+\alpha .
%\end{equation}
%Since $f\in \mathcal{S}^*_q$, by \cite[Theorem~1.15]{IMS90}, we have the Herglotz representation
%$$\frac{zf^{'}(z)}{f(z)}=1+\int_{|\sigma|=1}\sigma z F_q^{'}(\sigma z)\rm{d}\mu(\sigma),
%$$
%where
%$$F_q(z)=\ds \sum_{n=1}^\infty \frac{2\ln q}{q^n-1}z^n, \quad z\in \D .
%$$
%In view of the relation (\ref{thm7-eqn}), we obtain
%$$\frac{zg^{'}(z)}{g(z)}=(1-\alpha)\left(1+\int_{|\sigma|=1}\sigma z F_{q,\alpha}^{'}(\sigma z)
%\rm{d}\mu(\sigma)\right)+\alpha.
%$$
%On simplification, we conclude the representation
%$$\frac{zg^{'}(z)}{g(z)}=1+\int_{|\sigma|=1}\sigma z F_{q, \alpha}^{'}(\sigma z)\rm{d}\mu(\sigma)
%$$
%where
%$$F_{q,\alpha}(z)=(1-\alpha)F_q(z)=\ds \sum_{n=1}^\infty (1-\alpha)\frac{2\ln q}{q^n-1}z^n, \quad z\in \D .
%$$
This completes the proof of our theorem.
\end{proof}

%(see Figure~\ref{Gx}, when $z=x$ is real).

%\begin{figure}[H]
%\begin{minipage}[b]{0.5\textwidth}
%\includegraphics[width=7cm]{G1by21by2x.pdf}
%\centering{Graph of $G_{1/2,1/2}(x)$, $0\le x<1$} %\label{G1by21by2x}
%\end{minipage}
%\begin{minipage}[b]{0.45\textwidth}
%\includegraphics[width=7cm]{G5by61by2x.pdf}
%\centering{Graph of $G_{5/6,1/2}(x)$, $0\le x<1$} %\label{G5by61by2x}
%\end{minipage}
%\vskip 0.3cm
%\caption{Graphs of the real functions $G_{q,\alpha}(x)$ for $0\le x<1$.}\label{Gx}
%\end{figure}

\begin{figure}[H]
\begin{minipage}[b]{0.5\textwidth}
\includegraphics[width=6cm,height=8.15cm]{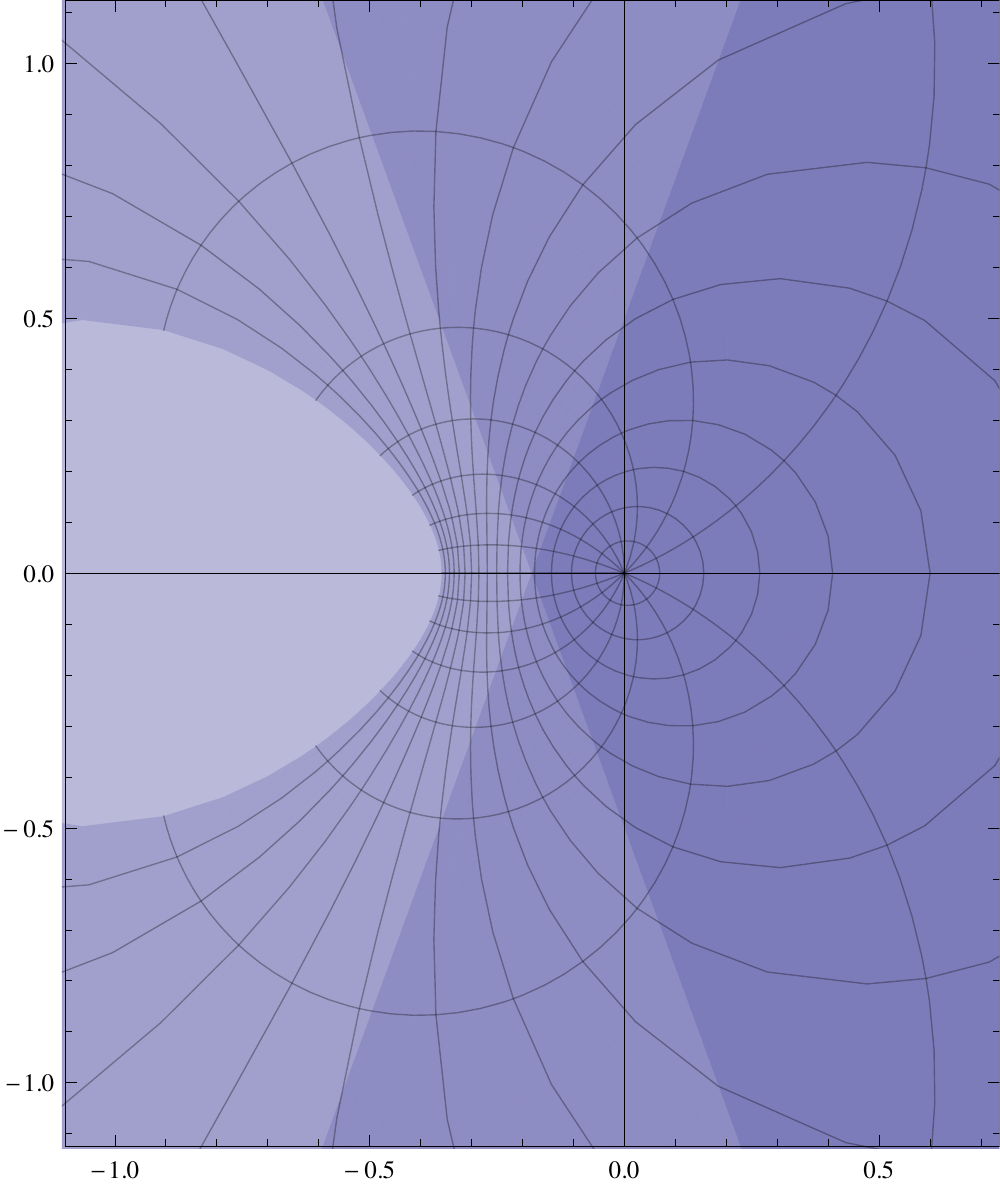}
\vskip 0.1cm \hskip 0.6cm
Graph of $G_{1/2,1/2}(z)$, $|z|<1$ %\label{G1by21by2z}
\end{minipage}
\begin{minipage}[b]{0.4\textwidth}
\includegraphics[width=6.5cm]{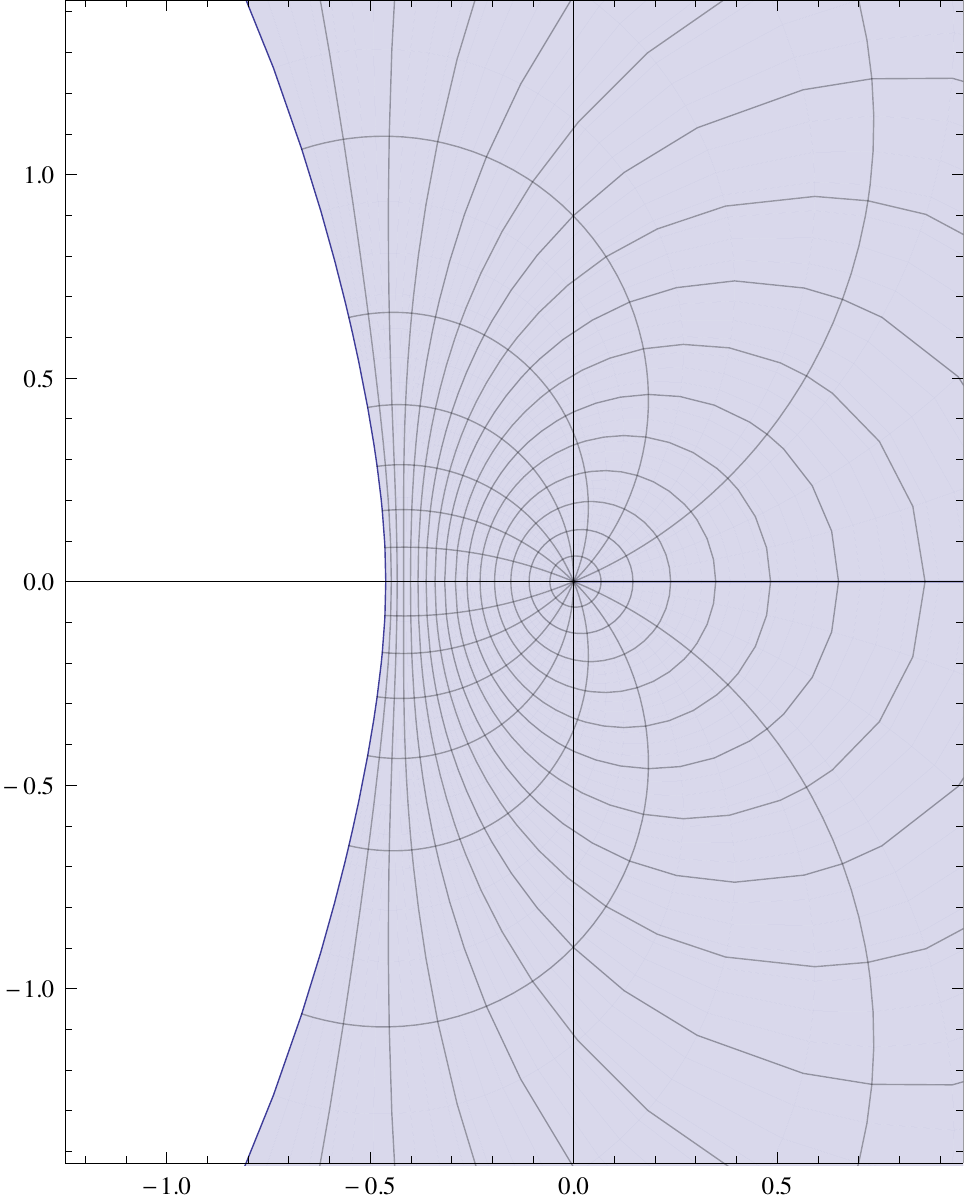}
\vskip 0.1cm \hskip 1cm
Graph of $G_{5/6,1/2}(z)$, $|z|<1$ %\label{G5by61by2z}
\end{minipage}
\vskip 0.3cm
\caption{Graphs of the complex functions $G_{q,\alpha}(z)$ for $|z|<1$.}\label{Gz}
\end{figure}

\begin{proof}[\bf Proof of Theorem~\ref{sec2-thm7}]
For $0<q<1$ and $0\le \alpha<1$, let $G_{q, \alpha}$ be defined by (\ref{MainThm2:eq}).
Geometry of the mapping $G_{q, \alpha}$ is described in Figure~\ref{Gz} for different ranges
over the parameters $q$ and $\alpha$.
As a special case to Theorem~\ref{thm2}, when the measure has a unit mass, it is clear that 
$G_{q, \alpha}\in \mathcal{S}^*_q(\alpha)$. Let $f\in \mathcal{S}^*_q(\alpha)$.
Then by Theorem~\ref{thm1}, there exist a function $\rho$ such that 
$\rho(f)(z)=h(z)=\ds\left(\frac{f(qz)}{f(z)}-\alpha q\right)/(1-\alpha)\in B_q ^0$. 
Since $h\in B_q ^0$, $g(z)=((1-\alpha)h(z)+\alpha q)/(1-\alpha(1-q))\in B_{q,\alpha} ^0 $. 
By Lemma~\ref{lm}, $g(z)$ has the representation (\ref{eq4}) and on solving we get,
\begin{equation}\label{eq8}
\frac{f(qz)}{f(z)}=(1-\alpha(1-q))\exp\left\{\left(\ln\frac{q}{1-\alpha(1-q)}\right)p(z)\right\}.
\end{equation}
Define the function $\phi(z)=\Log\{f(z)/z\}$ and set
\begin{equation}\label{eq9}
\phi(z)=\Log\frac{f(z)}{z}=\sum_{n=1}^\infty \phi_n z^n.
\end{equation}
On solving, we get 
%using (\ref{eq8}) and (\ref{eq9}), we get
$$
\ln\frac{q}{1-\alpha(1-q)}+\phi(qz)=\phi(z)+\left(\ln\frac{q}{1-\alpha(1-q)}\right)p(z).
$$
This implies
$$
\phi_n=p_n\left(\ln\frac{q}{1-\alpha(1-q)}\right)/(q^n-1).
$$
Since $|p_n|\le 2$, we have
$$
|\phi_n|\le \frac{(-2)\left(\ln \frac{q}{1-\alpha(1-q)}\right)}{1-q^n}.
$$
From this inequality, together with the expression of $G_{q, \alpha}(z)$ and (\ref{eq9}), the conclusion follows.
\end{proof}

\section{Order of $q$-starlikeness of $z\Phi[a,b;c;q,z]$}
The basic hypergeometric function is associated with the
{\em Watson symbol} $(a;q)_n$ (also called {\em the $q$-shifted factorial}), $n\ge 0$. The Watson symbol is defined by
$$(a;q)_n=(1-a)(1-aq)(1-aq^2)\cdots (1-aq^{n-1})=\prod_{k=0}^\infty \frac{1-aq^k}{1-aq^{k+n}}, \quad (a;q)_0=1
$$
for all real or complex values of $a$. 
In the unit disk $\mathbb{D}$, the {\em basic hypergeometric series} (also called {\em Heine hypergeometric series}) 
is defined by
$$\sum_{n=0}^\infty \frac{(a;q)_n(b;q)_n}{(c;q)_n(q;q)_n}z^n
 = 1+\frac{(1-a)(1-b)}{(1-c)(1-q)}z+\frac{(1-a)(1-aq)(1-b)(1-bq)}{(1-c)(1-cq)(1-q)(1-q^2)}z^2+\cdots,
$$
where $|q|<1$ and $a,b,c$ are real or complex parameters with $(c;q)_n\neq 0$, is convergent. 
The corresponding 
functions are denoted by $\Phi[a,b;c;q,z]$ and are referred to as the 
{\em basic (or Heine) hypergeometric functions} \cite{AAR99,Sla66}. 
The function $z\Phi[a,b;c;q,z]$ is called the {\em shifted basic hypergeometric function.}
The limit
$$\lim_{q\to 1^-}\frac{(q^a;q)_n}{(q;q)_n}=\frac{a(a+1)\cdots (a+n-1)}{n!}
$$
says that, with the substitution $a\mapsto q^a$, the Heine hypergeometric function takes to the well-known 
Gauss hypergeometric function $F(a,b;c;z)$ when $q$ approaches $1^-$.
For basic properties of Heine's hypergeometric series, readers may refer to \cite{GR90}.
K\"{u}stner \cite{Kus02} studied the order of starlikeness for functions $f\in \mathcal{A}$ 
by introducing the quantity
$$\sigma(f):=\inf_{z\in \mathbb{D}}{\rm Re}\,\left(\frac{zf'(z)}{f(z)}\right)\in [-\infty,1].
$$
Certainly, for the identity function $\sigma(f)=1$. K\"{u}stner found in \cite{Kus02} that
the well-known Gauss shifted hypergeometric functions $zF(a,b;c;z)$ have the order of 
starlikeness $-\infty$ for certain constraints on the real parameter $a,b,c$.
For $f\in\mathcal{A}$, let us now define the quantity
$$\sigma_q(f)=\inf_{z\in \D} \real \left(\frac{z(D_qf)(z)}{f(z)}\right) \in [-\infty,1].
$$
We call this quantity as the order of $q$-starlikeness of the function $f$.
Clearly, $\sigma_q(f)=1$ for $f(z)=z$. Note that $\lim_{q\to 1}\sigma_q(f)=\sigma(f)$.
We consider the shifted basic 
hypergeometric functions introduced by Heine \cite{Hei46} and study its $q$-analog of 
order of starlikeness.
Basic background knowledge on the order of starlikeness of the well-known Gauss hypergeometric
functions can be found in \cite{HPV10,Kus02,MM90,Pon97,PV01,Sil93}. 
 
In Theorem~\ref{sec3-thm1} we find the order of $q$-starlikeness of shifted basic hypergeometric functions $z\Phi[a,b;c;q,z]$.

\begin{theorem}\label{sec3-thm1}
Let $a,b,c$ be non-negative real numbers with $0<1-aq< 1-cq$ and $0<1-b< 1-c$. 
For $0<q<1$  and $r \in (0,1]$, 
the function $z\mapsto z\Phi[a,b;c;q,rz]$ has the order of $q$-starlikeness
$$\sigma_q(z\Phi[a,b;c;q,rz])=1+\rho q\frac{(1-a)(1-b)}{(1-c)(1-q)}\,\frac{\Phi[aq,bq;cq;q,\rho]}
{\Phi[a,b;c;q,\rho]}
$$
where 
$$ \rho=-r \mbox{ if } \frac{q(1-a)}{a(1-q)}=:s > 0 \mbox{ and } \rho=r \mbox{ if } s < 0.
$$
In particular, we have
$$1+\frac{s\rho}{1-\rho}\le \sigma_q(z\Phi[a,b;c;q,rz]) \le 1+\frac{\rho s(1-b)}{2(1-c)}.
$$
\end{theorem}

\begin{remark}
The case $s<0$ with $r=1$ in Theorem~\ref{sec3-thm1} is considered in the limiting sense.
In this case, the lower bound $1+\displaystyle\frac{sr}{1-r}$ is equal to $-\infty$.
\end{remark}

\begin{proof}[\bf Proof of Theorem~\ref{sec3-thm1}]
Set $\Phi(z)=\Phi[a,b;c;q,z]$ and $f(z)=z \Phi(z)$. 
Now, by (\ref{sec1-eqn1}) we have
\begin{eqnarray*}
(D_qf)(z)&=&\frac{\Phi(z)-q\Phi(qz)}{1-q}\\
&=&\frac{\Phi(z)-q\Phi(z)+q\Phi(z)-q\Phi(qz)}{1-q}\\
&=&\frac{\Phi(z)(1-q)+q(\Phi(z)-\Phi(qz))}{1-q}\\
&=&\Phi(z)+zq(D_q \Phi)(z).
\end{eqnarray*}
Hence,
%\begin{equation}\label{eqn0.1}
%\frac{z(D_qf)(z)}{f(z)}=1+zq\frac{(D_q \Phi)(z)}{\Phi(z)}.
%\end{equation}
%Now, as obtained in \cite{AS14}, we have
%\begin{eqnarray*}
%\Phi(z)-\Phi(qz) 
%&=&\frac{(1-a)(1-b)}{(1-c)}z\,\Phi[aq,bq;cq;q,z].
%\end{eqnarray*}
%This gives
%$$(D_q\Phi)(z)=\frac{(1-a)(1-b)}{(1-c)(1-q)}\,\Phi[aq,bq;cq;q,z] .
%$$
%Equation (\ref{eqn0.1}) yields
\begin{equation}\label{eqn0.2}
w=\frac{z(D_qf)(z)}{f(z)}=1+zq\frac{(1-a)(1-b)}{(1-c)(1-q)}\,\frac{\Phi[aq,bq;cq;q,z]}
{\Phi[a,b;c;q,z]},
\end{equation}
where the last equality holds by \cite[1.12(ii), pp. 27]{GR90}. Recall the difference equation stated in \cite{AS14}, which is equivalent to
\begin{equation}\label{sec3-eq}
\frac{\Phi[aq,bq;cq;q,z]}{\Phi[a,b;c;q,z]}=\frac{(1-c)}{a(1-b)z} \left[\frac{\Phi[aq,b;c;q,z]}{\Phi[a,b;c;q,z]}-1\right] .
\end{equation}
Substituting this ratio in (\ref{eqn0.2}), we get
$$w = 1+s\left[\frac{\Phi[aq,b;c;q,z]}{\Phi[a,b;c;q,z]}-1\right]
= 1-s+s\frac{\Phi[aq,b;c;q,z]}{\Phi[a,b;c;q,z]},
$$
where $s$ is defined in the statement of our theorem with $q \in (0,1)$.
It follows from \cite{AS14} that $w$ has an integral representation
\begin{equation}\label{eqn0.3}
w=1-s+s\int_0^1 \frac{1}{1-tz}\mbox{d}\mu(t),
\end{equation}
%or,
%$$|w|^2= 1+s^2-2s+s^2\left |\int_0^1 \frac{{\rm d}\mu(t)}{1-tz}\right |^2+2s \real 
%\int_0^1 \frac{{\rm d}\mu(t)}{1-tz} -2s^2 \real \int_0^1 \frac{{\rm d}\mu(t)}{1-tz}
%$$
with the non-negative real numbers $a,b,c$ satisfying the conditions 
$0\le 1-aq \le 1-cq$ and $0<1-b<1-c$.
Now, for $ s>0$, $r\in (0,1]$ and from equation (\ref{eqn0.3}) it follows that the minimum of 
$\real w$ for $|z|\le r$ is attained at the point $z=-r$ and that the minimum is 
$1-\ds \frac{rs}{(1+r)}$ .
%and 
%$$ |w|^2 \le s^2\left(1+\frac{1}{(1-r)^2}-\frac{2}{(1+r)}\right)+\frac{2rs}{(1-r)}+1.
%$$
Secondly, for $s<0$, $r\in (0,1]$ and from equation (\ref{eqn0.3}), 
it follows that the minimum of $\real w$ for $|z|\le r$ is attained at 
the point $z=r$ and that the minimum is $1+\ds \frac{rs}{(1-r)}$.
% This, together with (\ref{eqn0.1}) proves the theorem.
% and 
%$$|w|^2 \le s^2\left(1+\frac{1}{(1-r)^2}-\frac{2}{(1+r)}\right)-\frac{2rs}{(1+r)}+1.
%$$
This in combination with (\ref{eqn0.2}), yields the order of $q$-starlikeness of
$z\Phi[a,b;c;q,rz]$.

The upper estimate for ${\rm Re}\,w$
follows from (\ref{eqn0.2}) and an integral representation of the ratio 
${\Phi[aq,bq;cq;q,z]}/{\Phi[a,b;c;q,z]}$ obtained in \cite[Theorem~2.13]{AS14}. 
Hence, the conclusion of our theorem follows.
\end{proof}
%\hfill{$\Box$}

\begin{remark}
Making the substitutions $a\to q^a$, $b\to q^b$ and $c\to q^c$, and taking the limit
as $q\to 1^-$, we achieve the result of K\"ustner \cite[Theorem~1.1]{Kus02}. 
\end{remark}

\begin{remark}
If $f\in \mathcal{S}_q^*(\alpha)$, $0\le \alpha <1$, the order of $q$-starlikeness for $f$
can be equivalently defined by the quantity 
$$\sigma_q(f):=\inf_{z\in \mathbb{D}} \left \{\frac{1}{1+q}\left(1+q \real w - 
\sqrt{(1-q \real w)^2 -2(1-q)\real w+(1-q^2)|w|^2}\right)\right \},
$$
where $w=\ds\frac{z}{f(z)}(D_qf)(z)$.
Also, one can prove that $\sigma_q(f)\ge \alpha$.
Squaring the inequality given in (\ref{main=defn}) on both the 
sides, we get
$$f\in\mathcal{S}_q^*(\alpha) \iff \alpha^2(1+q)-2\alpha(1+q{\rm Re}\,w)
+2{\rm Re}\,w-(1-q)|w|^2\ge 0.
$$
Solving the inequality for $\alpha$, we obtain the required 
order, $\sigma_q(f)$, of $q$-starlikeness of functions $f\in\mathcal{S}_q^*(\alpha)$.
Since for all $f\in \mathcal{S}_q^*(\alpha)$, $\real w\ge \alpha$ and the lower bound 
$\alpha$ is attained whenever $\real w=|w|$, we get
$$\sigma_q(f)=\inf_{z\in \D} \real w=\alpha.
$$
\end{remark}

\section{Concluding remarks}
At the beginning of the last century, studies on $q$-difference equations appeared in intensive 
works especially by Jackson \cite{Jac10}, Carmichael \cite{Car12}, Mason \cite{Mas15}, Adams \cite{Ada29},
Trjitzinsky \cite{Trj33}, and later by others such as Poincar\'{e}, Picard, Ramanujan. Unfortunately,
from the thirties up to the beginning of the eighties only non-significant interest in this area was
investigated. Recently some research in this topic is carried out by Bangerezako \cite{Ban08}; 
see also references therein for other related work.
Research works in connection with function theory and $q$-theory together were first introduced 
by Ismail and et. al. \cite{IMS90}. Later it is also studied in \cite{Ro92,RS12,SS14,AS14}. 
Since only few work have been carried out in this direction, as indicated in \cite{AS14}, 
there are a lot can be done. For instance, $q$-analog of convexity 
of analytic functions in the unit disk and even more general in arbitrary 
simply connected domains may be interesting for researchers in this field.
Recently, the concept of $q$-convexity for basic hypergeometric functions is 
considered in \cite{BS14}.
Bieberbach conjecture problem for $q$-close-to-convex functions is 
estimated optimally in a recent paper \cite{SS14}. In fact sharpness of this result is 
still an open problem and concerning this, a conjecture is stated there.

\vskip 1cm
\noindent
{\bf Acknowledgement.} The work of the first author is supported by University
Grants Commission, New Delhi (grant no. F.2-39/2011 (SA-I)). The authors would like to thank the referee for his/her 
careful reading of the manuscript and valuable suggestion.


\begin{thebibliography}{99}

\bibitem{Ada29} {C. R. Adams}, 
On the linear ordinary $q$-difference equation, 
{\em Am. Math. Ser. II,} {\bf 30} (1929), 195--205.

\bibitem{AS14}{S. Agrawal and S. K. Sahoo}, 
Geometric properties of basic hypergeometric functions, 
{\em J. Difference Equ. Appl.}, {\bf 20} (11) (2014), 1502--1522.

\bibitem{AAN12} {B. Ahmad, A. Alsaedi, and S. K. Ntouyas}, A study of second order 
$q$-difference equations with boundary conditions, 
{\em Adv. Difference Equ.}, {\bf 35} (1) (2012), 1--10.

\bibitem{And74} {G. E. Andrews}, 
Applications of basic hypergeometric functions,
{\em SIAM Rev.}, \textbf{16} (1974), 441--484.

\bibitem{AAR99} {G. E. Andrews, R. Askey} and {R. Roy}, 
{\em Special Functions}, Cambridge University Press, Cambridge, 1999.

%\bibitem{Ask92} {R. A. Askey},
%The world of $q$, {\em Centrum voor Wiskunde en Informatica} {\bf 5} (1992):4, 251--269.

\bibitem{Ban08} {G. Bangerezako}, 
{\em An Introduction to $q$-Difference Equations}, University of Burundi, Bujumbura, 2007, Preprint.

\bibitem{BS14} A. Baricz and A. Swaminathan, Mapping properties of basic hypergeometric functions, 
{\em J. Class. Anal.}, {\bf 5} (2) (2014), 115--128.

\bibitem{deB85} {L. de Branges},
A proof of the Bieberbach conjecture,
{\em Acta Math.,}  {\bf 154} (1-2) (1985), 137--152

\bibitem{Car12} {R. D. Carmichael}, 
The general theory of linear $q$-difference equations,
{\em Amer. J. Math.}, {\bf 34} (1912), 147--168.

\bibitem{Dur83}  {P. L. Duren,}
\emph{Univalent Functions}, Springer-Verlag, New York, 1983.

\bibitem{Ern02} {T. Ernst}, 
{\em The History of $q$-calculus and a New Method},
Licentiate Dissertation, Uppsala, 2001.

\bibitem{Fin88} {N. J. Fine}, 
{\em Basic Hypergeometric Series and Applications},
Mathematical Surveys No. 27, Amer. Math. Soc. Providence, 1988.

% \bibitem{Fin98} {R. J. Finkelstein},
% $q$-Uncertainty relations, {\em Internat. J. Modern Phys. A}, {\bf 13} (1998), no. 11, 1795--1803.

\bibitem{GR90} {G. Gasper and M. Rahman} ,
{\em Basic Hypergeometric Series}, Encyclopedia of Mathematics and its Applications,
{\bf 35}, Cambridge University Press, Cambridge, 1990.

\bibitem{Goo83} {A. W. Goodman},
{\em Univalent Functions}, Volume 1, Mariner Publishing Company, Inc., Florida, 1983. 

\bibitem{HPV10} {P. H\"{a}st\"{o}, S. Ponnusamy, and M. Vuorinen},
Starlikeness of the gaussian hypergeometric functions,
{\em Complex Var. Elliptic Equ.,} {\bf 55} (2010), 173--184.

\bibitem{Hei46} {E. Heine},
\"Uber die Reihe $1+\frac{(q^\alpha -1)(q^\beta -1)}{(q-1)(q^\gamma -1)}x+
\frac{(q^\alpha -1)(q^{\alpha +1} -1)(q^\beta -1)(q^{\beta +1} -1)}{(q-1)(q^2-1)
(q^\gamma -1)(q^{\gamma +1} -1)}x^2+\cdots $, 
{\em J. Reine Angew. Math.}, {\bf 32} (1846), 210--212.

\bibitem{IMS90}{M. E. H. Ismail, E. Merkes, D. Styer},
A generalization of starlike functions,
{\em Complex Variables}, {\bf 14} (1990), 77--84.

\bibitem{Jac10} {F. H. Jackson}, 
On $q$-definite integrals, 
{\em Quart. J. Pure and Appl. Math.}, {\bf 41} (1910), 193--203.

% \bibitem{Jac10-1} {F. H. Jackson}, 
% On $q$-difference equations, {\em American J. Math.}, {\bf 32} (1910), 305--314.

\bibitem{KC02} {V. Kac and P. Cheung}, 
{\em Quantum calculus},
Universitext, Springer-Verlag, New York, 2002.

\bibitem{Kir95} {A. N. Kirillov},
Dilogarithm identities,
{\em Progr. Theoret. Phys. Suppl.}, {\bf 118} (1995), 61--142.

%\bibitem{Koo94} {T. Koornwinder,} 
%Compact quantum groups and $q$-special functions, 
%European school of group theory, {\em Pitman Research Notes} {\bf 311}, 1994.

\bibitem{Kus02} {R. K\"ustner}, 
Mapping properties of hypergeometric functions and convolutions of starlike and convex functions of order $\alpha$, 
{\em Comput. Methods Funct. Theory}, {\bf 2} (2) (2002), 567--610.

\bibitem{Mas15} {T. E. Mason}, 
On properties of the solution of linear $q$-difference equations with entire fucntion coefficients,
{\em Amer. J. Math.}, {\bf 37} (1915), 439--444.

\bibitem{MM90} {S. S. Miller and P. T. Mocanu},
Univalence of Gaussian and confluent hypergeometric functions,
{\em Proc. Amer. Math. Soc.}, {\bf 110} (2) (1990), 333--342.

\bibitem{Pon97} {S. Ponnusamy},
Close-to-convexity properties of Gaussian hypergeometric functions,
{\em J. Comput. Appl. Math.}, {\bf 88} (1997), 327--337.

\bibitem{PV01} {S. Ponnusamy and M. Vuorinen},
Univalence and convexity properties for Gaussian hypergeometric functions,
{\em Rocky Mountain J. Math.}, \textbf{31} (2001), 327--353.

\bibitem{RS12} {K. Raghavendar and A. Swaminathan},
Close-to-convexity of basic hypergeometric functions using
their Taylor coefficients,
{\em J. Math. Appl.}, {\bf 35} (2012), 111--125.

\bibitem{Ro92} {F. R\o{}nning}, A Szeg\H{o} quadrature formula arising from
$q$-starlike functions: In {\em continued fractions and orthogonal functions} 
(Loen, 1992), 345--352, Dekker, New York. 
 
\bibitem{SS14}{S. K. Sahoo and N. L. Sharma},
On a generalization of close-to-convex functions,
{\em Ann. Polon. Math.}, {\bf 113} (1) (2015), 93--108.

\bibitem{Sil93} {H. Silverman},
Starlike and convexity properties for hypergeometric functions,
{\em J. Math. Anal. Appl.}, {\bf 172} (1993), 574--581.

\bibitem{Sla66}{L. J. Slater}, 
{\em Generalized Hypergeometric Functions}, Cambridge University Press, Cambridge, 1966.

%\bibitem{Tho69} {J. Thomae}, 
%Beitra\"{g}e zur Theorie der durch die Heinische Reihe $\ldots$, 
%{\em J. reine angew. Math.}, {\bf 70} (1869), 258--281.

%\bibitem{Tho70} {J. Thomae},
%Uber die h\"{o}heren hypergeometrischen Reihen, insbesondere die Reihe,
%{\em Math. Ann.}, {\bf 2} (1870), 427--444.

\bibitem{Trj33} {W. J. Trjitzinsky,} 
Analytic theory of linear $q$-difference equations, 
{\em Acta Math.}, {\bf 61} (1) (1933), 1--38.

\end{thebibliography}
\end{document}